\documentclass[reqno]{amsart}

\usepackage{color}
\usepackage[dvipsnames]{xcolor}

\usepackage{ifpdf}
\ifpdf 
    \usepackage[pdftex]{graphicx}   
    \usepackage[pdftex,     
            plainpages=false,   
            breaklinks=true,    
            colorlinks=true,
            linkcolor=red,
            citecolor=green,
            pdftitle=My Document
            pdfauthor=My Good Self
           ]{hyperref} 
\else 
    \usepackage{graphicx}       
    \usepackage{hyperref}       
\fi 

\hypersetup{colorlinks, linkcolor=red!50!black, citecolor=green!50!black, urlcolor=[rgb]{0.3, 0.2, 0.5}}

\usepackage{subfig}
\usepackage{glossaries}
\usepackage{glossary-mcols}

\usepackage{aurical}
\usepackage{amsfonts,amsmath}	
\usepackage{amssymb}
\usepackage{verbatim}
\usepackage{amsopn}
\usepackage[english]{babel}
\usepackage{amsthm}
\usepackage{enumerate}
\usepackage{mathrsfs}	
\usepackage{mathtools}
\usepackage{esint}
\usepackage{todonotes}

\usepackage{bm}
\usepackage{bbm}
\usepackage{caption}

\usepackage{tikz}
\usepackage{tikz-3dplot}
\usepackage{pgfplots}
\usepackage{graphicx}
\usepackage{subfig}

\usetikzlibrary{intersections,calc}

\date{\today}

\theoremstyle{definition} \newtheorem{definition}{Definition}[section]
\theoremstyle{definition} \newtheorem{remark}[definition]{Remark}
\theoremstyle{plain} \newtheorem{lemma}[definition]{Lemma}
\theoremstyle{plain} \newtheorem{proposition}[definition]{Proposition}
\theoremstyle{plain} \newtheorem{theorem}[definition]{Theorem}
\theoremstyle{plain} 
\theoremstyle{definition} 
\theoremstyle{plain} 
\theoremstyle{definition} 
\theoremstyle{plain}

\DeclareMathOperator{\sign}{sign}

\let\d\relax
\newcommand{\d}{\partial}
\DeclareMathOperator{\esssup}{ess\,sup}
\DeclareMathOperator{\dive}{div}

\DeclareMathOperator{\dist}{dist}

\DeclareMathOperator{\Lip}{Lip}

\newcommand{\R}{\mathbb{R}}
\newcommand{\Q}{\mathbb{Q}}
\newcommand{\N}{\mathbb{N}}

\newcommand{\e}{\varepsilon}
\newcommand{\eps}{\varepsilon}
\newcommand{\fhi}{\varphi}

\newcommand{\M}{\mathscr{M}}

\newcommand{\1}{\mathbbm 1}

\renewcommand{\b}{{\bm b}}

\renewcommand{\H}{\mathscr H}
\newcommand{\rest}{\llcorner}

\numberwithin{equation}{section} 

\theoremstyle{plain} \newtheorem*{theorem*}{Theorem}
\theoremstyle{plain} 
\theoremstyle{plain} \newtheorem*{mthm*}{Main Theorem}
\theoremstyle{plain} \newtheorem*{conjecture*}{Conjecture}
\theoremstyle{plain} 
\theoremstyle{plain} \newtheorem*{problem*}{Problem}

\title[]{Non-uniqueness of signed measure-valued solutions to the continuity equation in presence of a unique flow}

\author{Paolo Bonicatto}
\address{Departement Mathematik und Informatik, Universit\"at Basel, Spiegelgasse 1, CH-4051, Basel, Switzerland.}
\email{paolo.bonicatto@unibas.ch}

\author{Nikolay A. Gusev}
\address{Moscow Institute of Physics and Technology,
9 Institutskiy per., Dolgoprudny, Moscow Region, 141700;
RUDN University,
6 Miklukho-Maklay St, Moscow, 117198;
Steklov Mathematical Institute of Russian Academy of Sciences,
8 Gubkina St, Moscow, 119991;}
\email{n.a.gusev@gmail.com}

\begin{document}

\begin{abstract}
	We consider the continuity equation $\partial_t \mu_t + \dive(\b \mu_t) = 0$, where $\{\mu_t\}_{t \in \R}$ is a measurable family of (possibily signed) Borel measures on $\R^d$ and $\b \colon \R \times \R^d \to \R^d$ is a bounded Borel vector field (and the equation is understood in the sense of distributions). If the measure-valued solution $\mu_t$ is non-negative, then the following \emph{superposition principle} holds: $\mu_t$ can be decomposed into a superposition of measures concentrated along the integral curves of $\b$. For smooth $\b$ this result follows from the method of characteristics, and in the general case it was established by L. Ambrosio. A partial extension of this result for signed measure-valued solutions $\mu_t$ was obtained in \cite{AB}, where the following problem was proposed: does the superposition principle hold for signed measure-valued solutions in presence of unique flow of homeomorphisms solving the associated ordinary differential equation? 
    We answer to this question in the negative, presenting two counterexamples in which uniqueness of the flow of the vector field holds but one can construct non-trivial signed measure-valued solutions to the continuity equation with zero initial data.
    \\
	
	

\end{abstract}

\maketitle

\section{Introduction}

In this paper we consider the initial value problem for the continuity equation
\begin{equation}\label{PDE}
\begin{cases}
\d_t \mu_t + \dive (\b \mu_t) = 0,\\ 
\mu_0 = \overline{\mu}
\tag{PDE}
\end{cases}
\end{equation}
for finite Borel measures $\{\mu_t\}_{t\in[0,T]}$ on $\R^d$, where $\b\colon [0,T]\times \R^d \to \R^d$ is a given bounded Borel vector field, $T>0$ and $d\in\N$ and $\overline{\mu}\in \mathscr M(\R^d)$ is a given measure on $\R^d$. This class of measure-valued solutions arises naturally in the limit for weakly* converging subsequences of smooth solutions, and it appears in various applications including hyperbolic conservation laws, optimal transport and other areas, see e.g. \cite{BouchotJames98,AGS2008,BPRS}.

In this paper we study the relationship between uniqueness of solutions to \eqref{PDE} and uniqueness to the ordinary differential equation drifted by $\b$, i.e. 
\begin{equation}\label{ODE}
\frac{d}{dt}\gamma(t) = \b(t,\gamma(t)), \quad t\in(0,T), \tag{ODE}
\end{equation}
where $\gamma\in C([0,T]; \R^d)$.

Given a solution $\gamma\in C([0,T]; \R^d)$ of \eqref{ODE} one readily checks that
$\mu_t := \delta_{\gamma(t)}$ solves \eqref{PDE}, where $\delta_p$ denotes the Dirac measure concentrated at $p$. Therefore uniqueness for \eqref{PDE} implies uniqueness for \eqref{ODE}. Hence it is natural to ask whether the converse implication holds.

In the class of non-negative measure-valued solutions the answer to this question is positive,
and it was obtained in \cite{AGS2008} as a consequence of the so-called \emph{superposition principle}. In order to formulate this principle, we will say that a family of Borel measures $\{\mu_t\}_{t\in[0,T]}$ is \emph{represented by} a finite (possibly signed) Borel measure $\eta$ on $C([0,T]; \R^d)$ if
\begin{enumerate}
\item $\eta$ is concentrated on $\Gamma_\b$;
\item $(e_t)_\sharp \eta = \mu_t$ for a.e. $t$,
\end{enumerate}
where $e_t\colon C([0,T]; \R^d) \to \R^d$ is the so-called \emph{evaluation map} defined by $e_t(\gamma) := \gamma(t)$, $(e_t)_\sharp \eta$ denotes the image of $\eta$ under $e_t$, and $\Gamma_\b$ denotes the set of solutions of \eqref{ODE} (note the $\Gamma_\b$ is a Borel subset of $C([0,T];\R^d)$ by \cite[Proposition 2]{Bernard2008}).
For example, if $\gamma \in C([0,T]; \R^d)$ solves \eqref{ODE} then $\eta:= \delta_\gamma$ (as a measure on $C([0,T]; \R^d)$) represents the solution $\mu_t := \delta_{\gamma(t)}$ of \eqref{PDE}.

A straightforward computation shows that if $\{\mu_t\}_{t\in[0,T]}$ is represented by some (possibly signed) measure $\eta$ then $\mu_t$ solves \eqref{PDE}. In this case we will say that $\mu_t$ is a \emph{superposition solution} of \eqref{PDE}.
Clearly uniqueness for \eqref{ODE} implies uniqueness for \eqref{PDE} in the class of superposition solutions.
Indeed, by uniqueness for \eqref{ODE} the continuous mapping $e_0 \colon \Gamma_\b \to \R^d$ is injective,
hence $e_0^{-1}$ is Borel and thus $(e_0)_\sharp \eta = \mu_0$ is equivalent to $\eta = (e_0^{-1})_\sharp \mu_0$.

Therefore, when uniqueness holds for the Cauchy problem for \eqref{ODE}, uniqueness for the Cauchy problem for \eqref{PDE} holds in the class of measure-valued solutions if and only if \emph{any} measure-valued solution of such Cauchy problem is a superposition solution.

The superposition principle established in \cite{AGS2008} states that any \emph{non-negative} solution $\mu_t$ of \eqref{PDE} can be represented by some non-negative measure $\eta$ on $C([0,T];\R^d)$.
However, without extra assumptions this result cannot be extended to \emph{signed} solutions, because \eqref{PDE} can have a nontrivial signed solution even when $\Gamma_\b = \emptyset$ (see e.g. \cite{Gus18a} for the details).

Under Lipschitz bounds on the vector field $\b$ uniqueness for \eqref{PDE} within the class of signed measures is well known, see e.g. \cite[Prop. 8.1.7]{AGS2008}. Out of the classical setting, the first (positive) result is contained in \cite{Bahouri1994}, where the authors considered log-Lipschitz vector fields. Later on, in the paper \cite{AB}, the authors proved that the signed superposition principle holds provided that the vector field satisfies a quantitative two-sided diagonal Osgood condition. More precisely, in \cite{AB} the authors considered vector fields enjoying 
\begin{itemize}
  \item[(O)] \label{i-osgood}
  it holds 
  \begin{equation*}
  \vert \langle \b(x)-\b(y), x-y \rangle \vert \le  C(t)\Vert x-y\Vert \rho(\Vert x-y \Vert)\qquad \forall x,y \in \mathbb R^d, \, \forall t \in (0,T),
  \end{equation*}
  where $C\in L^1(0,T)$ and $\rho \colon [0,1) \to \mathbb [0,+\infty) $ is an Osgood modulus of continuity, i.e. a continuous, non-decreasing function with $\rho(0)=0$ and 
  $$
  \int_0^1 \frac{1}{\rho(s)} \, ds = +\infty. 
  $$
  \item[(B)] \label{i-bounded}
  $\b$ is uniformly bounded.
\end{itemize}
Their results is the following: 
\begin{theorem}[Thm. 1 in \cite{AB}] 
  If the vector field $\b$ satisfies (\hyperref[i-osgood]{O}) and (\hyperref[i-bounded]{B}), then there is uniqueness for \eqref{PDE} in the class of bounded signed measures, i.e. if $\mu_t$ is a solution of \eqref{PDE} such that $\vert \mu_t \vert (\R^d) \in L^\infty(0,T)$ then
  \begin{equation*}
  \mu_t = \bm X(t,\cdot)_\# \mu_0, \qquad  \, \forall t \in (0,T),
  \end{equation*}
  where $\bm X(t,\cdot)$ is the flow of $\b$, i.e. the unique map solving 
  \begin{equation*}
  \begin{cases}
  \partial_t \bm X(t,x) = \b(t,\bm X(t,x))  & t \in [0,T], x \in \R^d\\ 
  \bm X(0,x) = x & x \in \R^d
  \end{cases}
  \end{equation*}
\end{theorem}

Notice that the Osgood assumption (\hyperref[i-osgood]{O}) is an assumption on $\b$, and it is much stronger than an implicit assumption of uniqueness for \eqref{ODE}.
For a simple example one can consider e.g. (for $d=1$) $\b(t,x) = \1_{(-\infty,0]}(x) + 2\cdot\1_{(0,+\infty)}(x)$.
Moreover, according to a theorem of Orlicz \cite{orlicz} (see also \cite[Thm. 1]{Bernard2008}), in the space of all continuous vector fields $\b$ (equiped with the topology of the uniform convergence on compact sets) those fields for which the differential equation \eqref{ODE} has at least one non-uniqueness point is of first category: this shows that in the generic situation Lipschitz/Osgood conditions are \emph{not} necessary for uniqueness.

Let us mention some other generic uniqueness results for \eqref{PDE}.
The one-dimensional case was studied in \cite{BouchotJames98}, where uniqueness of signed measure-valued solutions was obtained under the assumption that $\b$ satisfies a \emph{one-sided Lipschitz condition}, i.e. there exists $\alpha\in L^1(0,T)$ such that $\d_x \b(t, x) \le \alpha(t)$ (in the sense of distributions). Still in $d=1$, uniqueness in the class of absolutely continuous (with respect to Lebesgue measure) solutions was obtained in \cite{Gus18b} for nearly incompressible vector fields. In the multi-dimensional case
uniqueness of absolutely continuous solutions was obtained in \cite{BiaBon17a} for
nearly incompressible vector fields with bounded variation.
For generic solutions, besides \cite{AB}, one can refer to \cite{clop}, where uniqueness within the signed framework is shown for vector fields having an Osgood modulus of continuity.

The generic uniqueness results mentioned above require some \emph{regularity} of $\b$
(e.g. some form of weak differentiability), but as discussed above
one can ask if uniqueness for \eqref{ODE} is sufficient for uniqueness for \eqref{PDE}.
In particular, a natural question (raised in \cite{AB}) is whether uniqueness for \eqref{PDE} (in the class of signed measures) holds in the presence of a (unique) flow of homeomorphisms solving \eqref{ODE}, without an explicit bound like (\hyperref[i-osgood]{O}) on the vector field. 
We show that the answer to this question in general is negative by constructing
two counterexamples of bounded vector fields $\b \colon [0,T]\times \R^d \to \R^d$
(for $d=1$ and $d=2$)
such that for \emph{any} $x\in \R^d$ only $\gamma(t)\equiv x$ ($\forall t\in [0,T]$) solves \eqref{ODE}
but \eqref{PDE} with zero initial condition has a non-trivial measure-valued solution $\{\mu_t\}_{t\in[0,T]}$:
\begin{enumerate}
  \item[(i)] for $d=1$ there exists a nontrivial $[t \mapsto \mu_t] \in L^1([0,T]; \M(\R))$ solving \eqref{PDE} with zero initial condition. However in this example $[t \mapsto \mu_t] \not\in L^\infty([0,T]; \M(\R))$.
  \item[(ii)] for $d=2$ there exists a nontrivial $[t \mapsto \mu_t] \in L^\infty([0,T]; \M(\R^2))$ solving \eqref{PDE} with zero initial condition.
\end{enumerate}

In the examples (i) and (ii) of the present paper the vector field $\b$ is only bounded, but not continuous. However all vector fields that satisfy (\hyperref[i-osgood]{O}) and (\hyperref[i-bounded]{B}) are continuous (see Proposition~\ref{p-o+b=cont}). 
It would therefore be interesting to understand whether
for continuous vector fields uniqueness for \eqref{ODE} implies uniqueness for \eqref{PDE}.

Note that our examples (i) and (ii)
are based on a one-dimensional vector field that does not have integral curves
and hence cannot be continuous (in view of Peano's theorem).
And in fact for $d=1$ it is possible to prove that if $\b$ is stationary and continuous then uniqueness for \eqref{ODE} implies uniqueness for \eqref{PDE} (see Proposition~\ref{p-1d-cont-uniq}).
It is interesting to note that such $\b$ can be very irregular and hence
one cannot apply to it any of the generic uniqueness results discussed earlier.

Let us also mention that (still for $d=1$) if $\b$ is continuous and for any $t$
the function $x\mapsto \b(t,x)$ is non-strictly decreasing then
uniqueness holds both for \eqref{PDE} (this follows from \cite{BouchotJames98}) and for \eqref{ODE}
(this can be shown directly: if $\gamma_1$ and $\gamma_2$
are integral curves of $\b$ such that $\gamma_1(0)=\gamma_2(0)$ and $\gamma_1(t)<\gamma_2(t)$
for all sufficiently small $t>0$ then $\partial_t(\gamma_1(t) - \gamma_2(t)) =\b(t,\gamma_1(t)) - \b(t, \gamma_2(t)) \ge 0$).

\section{Preliminaries}

In the following, we will denote by $\mathscr B(\R^d)$ the Borel $\sigma$-algebra on $\R^d$. We recall some basic definitions.

\begin{definition}
  A family $\{\mu_t\}_{t\in[0,T]}$ of Borel measures on $\R^d$ is called \emph{a Borel family} if for any $A\in \mathscr B(\R^d)$ the map $t \mapsto \mu_t(A)$ is Borel-measurable.
\end{definition}

The following propositions are well-known (see, e.g. \cite[Prop. 2.26 and (2.16)]{AFP}):

\begin{proposition}\label{p-sufficient-family-of-Borel-measures}
  If $\{\mu_t\}_{t\in[0,T]}$ is a family of Borel measures on $\R^d$ such that $t\mapsto \mu_t(A)$ is Borel for any open set $A\subset \R^d$ then $\{\mu_t\}_{t\in[0,T]}$ is a Borel family.
\end{proposition}

\begin{proposition}\label{p-family-of-Borel-measures-total-variation}
  If $\{\mu_t\}_{t\in[0,T]}$ is a Borel family then $\{|\mu_t|\}_{t\in[0,T]}$ also is a Borel family.
\end{proposition}

\begin{proposition}\label{p-integrate-wrt-family-of-Borel-measures}
  If $\{\mu_t\}_{t\in[0,T]}$ is a Borel family then for any bounded Borel function $g\colon [0,T]\times\R^d \to \R$
  the map $t\mapsto \int_{\R^d} g(t,x) \, d\mu_t(x)$ is Borel.
\end{proposition}

In what follows we will write that $[t\mapsto \mu_t] \in L^1((0,T); \mathscr M(\R))$ if 
$\{\mu_t\}_{t\in [0,T]}$ is a a Borel family and
\begin{equation*}
\int_0^T |\mu_t|(\R^d)\, dt < +\infty.
\end{equation*}
If, in addition,
\begin{equation*}
\esssup_{t \in [0,T]} |\mu_t|(\R^d) < +\infty,
\end{equation*}
then we will write $[t\mapsto \mu_t] \in L^\infty((0,T); \mathscr M(\R))$.

In view of Proposition~\ref{p-integrate-wrt-family-of-Borel-measures} the distributional formulation of the continuity equation is well-defined:

\begin{definition}
  A family
  $[t\mapsto \mu_t] \in L^1((0,T); \mathscr M(\R))$
  is called a measure-valued solution of \eqref{PDE}
  if for any $\fhi\in C^1_c([0,T)\times \R^d)$
  \begin{equation}\label{eq-continuity-distrib}
    \int_0^T \int_{\R^d} (\d_t\fhi + \b(t,x) \cdot \nabla_x \fhi (t,x)) \, d\mu_t(x) \, dt 
    = \int_{\R^d} \fhi(0,x)\, d\bar \mu(x).
  \end{equation}
\end{definition}

Even though the distributional formulation of the Cauchy problem for \eqref{PDE} is well-defined for $[t\mapsto \mu_t] \in L^1((0,T); \mathscr M(\R))$, it is much more natural in the class $[t\mapsto \mu_t] \in L^\infty((0,T); \mathscr M(\R))$, because in this class the initial condition can be understood in the sense of traces, considering a weak* continuous representative of $[t\mapsto \mu_t]$. More precisely, we have the following Proposition (for a proof see e.g. \cite[Chapter 1, Prop. 1.6]{tesiphd}).

\begin{proposition}[Continuous representative] \label{prop:trace} Let $\{\mu_t\}_{t\in[0,T]}$ be a Borel family of measures and assume $[t\mapsto \mu_t] \in L^\infty((0,T); \mathscr M(\R))$.  Then there exists a \emph{narrowly continuous} curve $[0,T] \ni t\mapsto \tilde{\mu}_t \in \mathscr M(\R)$ such that $\mu_t = \tilde{\mu}_t$ for a.e. $t \in [0,T]$. \end{proposition}

\section{Non-uniqueness in the class 
\texorpdfstring{$L^1((0,T); \mathscr M(\R))$}{L1(0,T;M(R1))}}

In this section we prove the following result:

\begin{theorem}\label{t-L1}
There exist $T>0$, a bounded Borel $\b \colon [0,T] \times \R \to \R$ and $[t\mapsto \mu_t] \in L^1((0,T); \mathscr M(\R))$ satisfying the following conditions:
\begin{enumerate}
  \item[(i)] $\b$ has only constant characteristics, i.e. $\gamma\in \Gamma_\b$ if and only if there exists $x\in \R$ such that $\gamma(t) = x$ for all $t\in[0,T]$; 
  \item[(ii)] $\{\mu_t\}_{t\in [0,T]}$ solves \eqref{PDE} with zero initial condition.
\end{enumerate}
\end{theorem}

\subsection{Auxiliary result}
We begin by the following auxiliary result: although it is well-known, we give a proof because some details will be used later.

\begin{lemma}\label{l-sets-P-and-N}
  There exist Borel sets $P, N \subset \R$ such that
  \begin{enumerate}
    \item $P \cap N = \emptyset$;
    \item $|\R \setminus (P \cup N)| = 0$;
    \item for any nonempty bounded open interval $I\subset \R$ it holds that $|I\cap P| > 0$ and $|I\cap N|>0$,
  \end{enumerate}
  where $|A|$ denotes the Lebesgue measure of $A\subset \R$.
\end{lemma}

\begin{proof}
  Let $\{q_k\}_{k\in \N}$ be the set of all rational numbers.
  Let $f_0(x):=1$ ($x\in \R$), $E_0 := \{0, 1\}$ and $\e_0:=1$.
  
  Consider $k\in \N$ and suppose that the set $E_{k-1}$, the number $\e_{k-1}>0$ and the function $f_{k-1}$ are already constructed. We assume that $E_{k-1}$ is finite, $E_{k-1} \cap \Q = \emptyset$, hence $(0,1) \setminus E_{k-1}$
  is a union of finitely many open intervals. We also assume that $f_{k-1}$ is either $+1$ or $-1$ on each of these intervals.
  
  Since $\dist(q_k, E_{k-1})>0$ there exists $\e_k>0$ such that 
  \begin{equation}\label{e-I-pm-eps-k}
  \e_k < 2^{-k} \e_{k-1},
  \end{equation}
  \begin{equation}
  \left(q_k - \e_k, q_k + \e_k\right) \subset \R\setminus E_{k-1},
  \end{equation}
  and moreover
  \begin{equation}\label{e-I-pm-I-k-bdry}
  q_k \pm \frac12 \e_k \notin \Q.
  \end{equation}
  
  We then define
  \begin{equation}\label{e-I-pm-I-k}
  I_k := \left(q_k - \frac12 \e_k, q_k + \frac12 \e_k\right) \subset \R\setminus E_{k-1}
  \end{equation}
  and
  \begin{equation}\label{eq:def_f_k}
  f_k(x) := \begin{cases}
  f_{k-1}(x), & x\notin \overline{I_k}, \\
  -f_{k-1}(x), & x\in I_k, \\
  0, & x\in \partial I_k
  \end{cases}
  \end{equation}
  and $E_k := E_{k-1}\cup \partial I_k$.
  It is easy to see that $E_k, \e_k$ and $f_k$ satisfy the same assumptions as $E_{k-1}, \e_{k-1}$ and $f_{k-1}$.
  Therefore we can construct inductively the sequence $\{E_k, \e_k, f_k\}_{k\in \N}$.
  
  Consider the set $R_k := \cup_{n=k+1}^\infty \overline{I_n}$ 
  on which the function $f_n$ ($n>k$) may differ from $f_k$. By \eqref{e-I-pm-eps-k}
  \begin{equation}\label{e-I-pm-R-k}
  |R_k| \le \sum_{n=k+1}^\infty \e_{n} = \sum_{n=k}^\infty \e_{n+1} < \sum_{n=k}^\infty 2^{-(n+1)} \e_{n}
  < \e_k \sum_{n=k}^\infty 2^{-(n+1)} \le \frac{1}{2} \e_k.
  \end{equation}
  For any $x \in I_0 \setminus R_k$ it holds that $f_{n}(x) = f_k(x)$ for all $n > k$.
  Since $\e_k \to 0$ as $k\to \infty$, we conclude that $f_k$ converges a.e. to some function $f\colon \R \to \R$ as $k \to \infty$.
  Moreover, on the complement of Lebesgue negligible set
  $\bigcap_{k\in\N} R_k$ the function $f$ by construction takes only the values
  $\pm 1$.
  We therefore set 
  \begin{equation}
  P := f^{-1}(\{+1\}), \quad N := f^{-1}(\{-1\}).
  \end{equation}
  
  Consider an arbitrary nonempty bounded open $I \subset \R$. There always exists a nonempty open interval $J$ such that $\overline J \subset I$. Since $J$ contains infinitely many rationals and $\e_k\to 0$ as $k\to \infty$, there exists $k_0 \in \N$ such that $(q_k - \e_k, q_k + \e_k) \subset I$ whenever $k>k_0$. 

  Without loss of generality let us assume that $f_{k-1} = +1$ on $(q_k - \e_k, q_k + \e_k)$ (the argument is the same when this value is $-1$).
  Hence by construction
  \begin{equation*}
  f_k(x) = 
  \begin{cases}
  f_{k-1}(x) = +1, & x \in (q_k - \e_k, q_k + \e_k) \setminus \overline{I_k}, \\
  -f_{k-1}(x) = -1, & x \in I_k.
  \end{cases}
  \end{equation*}
  Ultimately, by \eqref{e-I-pm-R-k} the function $f$ may differ from $f_k$ only on the set $R_k$ and $|R_k|< \frac12 \e_k$.
  Therefore $|I\cap P| \ge |(q_k - \e_k, q_k + \e_k)| - |R_k| \ge \eps_k$
  and $|I\cap N| \ge |(q_{k+1} - \e_{k+1}, q_{k+1} + \e_{k+1})| - |R_{k+1}| \ge \eps_{k+1}$.
\end{proof}

\subsection{The construction of the counterexample}

Given the sets $P,N \subset \R$ constructed in Lemma \ref{l-sets-P-and-N} we now set 
\begin{equation}\label{eq:def_f}
f(\tau) := 2+\int_0^\tau (\1_P(r) - \1_N(r)) \, dr \quad \text{ and } \quad
F(\tau) := (f(\tau), \tau)
\end{equation}
where $\tau \in [0,1]$. 
Since the derivative of $f$ is equal to $\1_P - \1_N$ a.e.,
for convenience we denote $f' := \1_P - \1_N$.

We now set $T:=4$ and define 
\begin{equation}\label{e-def-b}
\b(t,x) := \1_{F[0,1]}(t,x) \cdot \frac{1}{f'(x)}
\quad\text{and}\quad
\tilde\mu_t := \sum_{x\in f^{-1}(t)} \sign(f'(x)) \delta_x.
\end{equation}
By definition $\b$ is Borel and bounded.
Moreover by the area formula $\{\tilde\mu_t\}_{t\in[0,T]}$ is
a measurable family of Borel measures.

\begin{lemma}
For $\b$ and $\tilde\mu_t$ defined above
\begin{equation*}
\d_t \tilde\mu_t + \dive (\b \tilde\mu_t) = - \delta_{F(1)} + \delta_{F(0)}
\qquad \text{in } \quad \mathscr D'((0,T)\times \R).
\end{equation*}
\end{lemma}

\begin{proof}
Using the area formula (since $f([0,1]) \subset (0,T)$) we get
\begin{align*}
\int_0^T\left( \int_{\R}\left( \d_t \fhi + \b \, \d_x \fhi \right)\, d\tilde\mu_t(x)\right)\, dt &=
\int_0^T\left( \sum_{x\in f^{-1}(t)}\left( (\d_t \fhi)(t,x) + \frac{1}{f'(x)} \, (\d_x \fhi)(t,x) \right) \frac{f'(x)}{|f'(x)|}\right) \,dt
\\&=
\int_0^1\left( (\d_t \fhi)(f(x),x) + \frac{1}{f'(x)} \, (\d_x \fhi)(f(x),x) \right) f'(x) \,dx
\\&=
\int_0^1\left( f'(x) (\d_t \fhi)(f(x),x) + (\d_x \fhi)(f(x),x) \right) \,dx
\\&=
\int_0^1\d_x \left( \fhi(f(x),x) \right) \,dx
=\fhi(f(1),1) - \fhi(f(0),0). \qedhere
\end{align*}
\end{proof}

To get rid of the defect $- \delta_{F(1)} + \delta_{F(0)}$ we simply add to $\tilde\mu_t$ solutions concentrated on constant in time trajectories (since $\b$ is 0 outside $F([0,1])$). More precisely, one readily checks that
\begin{equation*}
\mu_t := \tilde \mu_t + \1_{[f(1),+\infty)}(t) \, \delta_1 - \1_{[f(0),+\infty)}(t) \, \delta_0
\end{equation*}
solves \eqref{PDE}.

\begin{figure}
  \def\svgwidth{.7\columnwidth}
  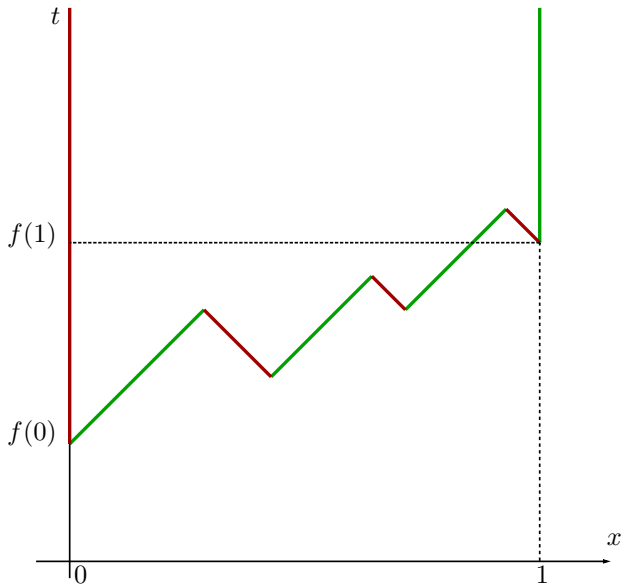 
  \caption{Graph of the function $t=f(x)$ (approximation step). At each $t \in [0,T]$ the measure $\tilde{\mu}_t$ is a superposition of Dirac masses with weight given by $\sign f^\prime(x)$, where $x \in f^{-1}(t)$ (notice the red/green parts).}
  \label{fig:}
\end{figure}

To conclude the proof of Theorem~\ref{t-L1}, it remains to study the integral curves of $\b$. This issue is addressed in the following Lemma:
\begin{lemma}\label{l-uniqueness-of-characteristics-1D}
  For any $(t,x)$ there exists a unique characteristic of $\b$ passing through $x$.
\end{lemma}
\begin{proof}
  Clearly points $\gamma(t)=x$, $t>0$, are characteristics of $\b$. Since the image of $[0,1]$ under $F$ is closed,
  $b$ vanishes identically in a neighbourhood of any $(t,x) \notin F([0,1])$. Therefore for $(t,x) \notin F([0,1])$
  the claim is trivial.
  
  Hence it is sufficient to prove that any characteristic $\gamma=\gamma(t)$ of $\b$ intersects $F([0,1])$ at most in one point.
  We argue by contradiction: suppose there exist $x<y$ such that
  \begin{equation}
  \gamma(f(x)) = x \quad \text{and} \quad \gamma(f(y)) = y.
  \end{equation}
  Since $\gamma' = \b(t, \gamma)$ and $\|\b\|_\infty \le 1$ it holds that
  \begin{equation}\label{e-gamma-is-in-the-vertical-cone}
  |x-y| = |\gamma(f(x)) - \gamma(f(y))| \le |f(x)-f(y)|
  \end{equation}
  On the other hand, by properties of the sets $P$ and $N$
  \begin{equation}\label{e-f-is-in-the-horizontal-cone}
  |f(y) - f(x)| \le \int_x^y |\1_P(z) - \1_N(z)| \, dz < |x-y|.
  \end{equation}
  The inequalities \eqref{e-gamma-is-in-the-vertical-cone} and \eqref{e-f-is-in-the-horizontal-cone} are not compatible,
  hence the proof is complete.
\end{proof}

Therefore we have constructed a vector field $\b$ for which the characteristics are unique, but there exists a nontrivial signed solution of the CE. Using a minor modification of the present construction one can construct a similar example of $(\mu_t, b)$ having compact support in spacetime. 

\begin{remark}
  The constructed solution $\{\mu_t\}$ is not a superposition solution (see Introduction).
\end{remark}

\begin{remark} As we have already remarked in Section 2, the distributional formulation of the Cauchy problem for \eqref{PDE} is well-defined for $[t\mapsto \mu_t] \in L^1((0,T); \mathscr M(\R))$ but it is best suited  in the class $[t\mapsto \mu_t] \in L^\infty((0,T); \mathscr M(\R))$, because of Proposition \ref{prop:trace}. Unfortunately for the present construction it holds that
  $[t\mapsto \mu_t] \notin L^\infty((0,T); \mathscr M(\R))$, as we show below in the remaining part of this section.
  However in the next section we present a two-dimensional analog of Theorem~\ref{t-L1} with the additional property that
  $[t\mapsto \mu_t] \in L^\infty((0,T); \mathscr M(\R^2))$.
\end{remark}

\begin{lemma}\label{l-somewhere-monotone}
Let $g\in \Lip((0,1))$ be such that $g'\ne 0$ a.e.
and
\begin{equation*}
\esssup_{t\in \R} \#(g^{-1}(t)) < \infty.
\end{equation*}
Then there exists a nonempty open interval $I\subset (0,1)$
such that $g$ is strictly monotone on $I$.
\end{lemma}

\begin{proof}
Let ${C}$ denote the set of points $x\in (0,1)$ where $g$
is not differentiable or $g'(x)=0$.
By the assumptions (and Rademacher's theorem) ${C}$ has measure zero.
Then by the area formula
\begin{equation*}
0=\int_{C} |g'(x)| \, dx= \int_{g({C})} \#(g^{-1}(t))\, dt,
\end{equation*}
hence $g({C})$ has zero Lebesgue measure (since $\#(g^{-1}(t)) \ge 1$ for all $t\in g({C})$).

Let
\begin{equation*}
M:= \esssup_{t\in \R} \#(g^{-1}(t)).
\end{equation*}
Since for any $t\in \R$ we have $\#(g^{-1}(t)) \in \N \cup \{0\}$ ,
there exists a set $R \subset \R$ with strictly positive measure
such that $\#(g^{-1}(t)) = M$ for all $t\in R$.
In particular, we can take $t\in R \setminus g({C})$.
Then $g^{-1}(t) = \{x_1, x_2, \ldots, x_M\}$
and $g'(x_i) \ne 0$. Hence there exist disjoint open intervals $I_i$
containing $x_i$ such that $g(\cdot)-t$ has different signs on $\partial I_i$, where $i=1,2,\ldots, M$.

\begin{figure}
	\def\svgwidth{.7\columnwidth}
	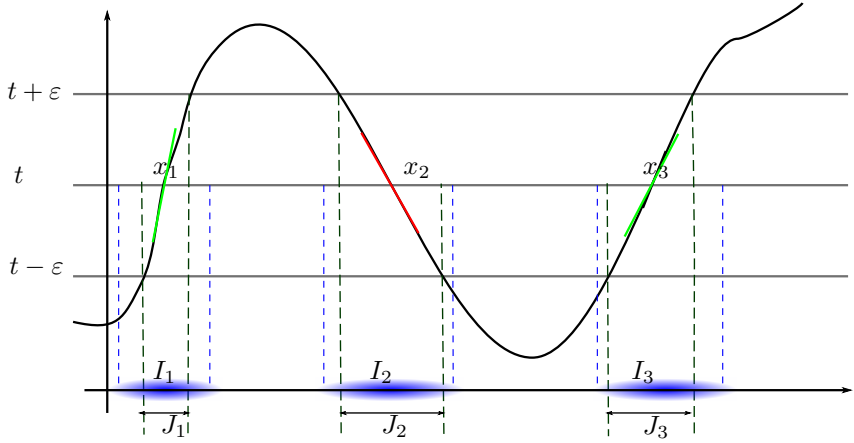 
	\caption{Situation described in the proof of Lemma \ref{l-somewhere-monotone}. The intervals $I_i$ are depicted in blue.}
	\label{fig:proof}
\end{figure}

Using continuity of $g$ we can always find an $\eps>0$ such that
$[t-\eps, t+\eps] \subset \bigcap_{i=1}^M g(I_i)$.
Hence, by the intermediate value property we can find nonempty open intervals $J_i \subset I_i$ (with $x_i\in J_i$) such that $g(\d J_i) =\{t-\eps, t+\eps\}$ for each $i\in 1,2\ldots,M$.

By the intermediate value property for each $\tau \in [t-\eps, t+\eps]$ we have 
\begin{equation}\label{e-lower-bound-on-N-of-preim}
\#(g^{-1}(\tau)\cap J_i)\ge 1, \qquad i\in 1,\ldots,M.
\end{equation}
On the other hand for \emph{all} $\tau \in [t-\eps, t+\eps]$ we have
\begin{equation}\label{e-bound-on-N-of-preim}
\sum_{i=1}^M \#(g^{-1}(\tau)\cap J_i) \le M.
\end{equation}
Indeed, by the definition of $M$ the estimate \eqref{e-bound-on-N-of-preim} holds for a.e. $\tau$, and
if it fails for some $\tau$, then at least for some $i$
it holds that $\#(g^{-1}(\tau)\cap J_i) \ge 2$.
Since $g'\ne 0$ a.e., by the intermediate value property
this implies existence of $\xi>0$ such that $\#(g^{-1}(s)\cap J_i) \ge 2$ for all $s\in [\tau, \tau + \xi)$ (or all $s\in (\tau-\xi,\tau]$), and in view of \eqref{e-lower-bound-on-N-of-preim} this clearly contradicts the definition of $M$.

From the estimates \eqref{e-lower-bound-on-N-of-preim} and \eqref{e-bound-on-N-of-preim} we conclude that for all $\tau\in [t-\eps, t+\eps]$ it holds that
\begin{equation*}
\#(g^{-1}(\tau)\cap J_i)= 1, \qquad i\in 1,\ldots,M.
\end{equation*}
Therefore for each $i\in 1,\ldots,M$ the function $g$
is injective on $J_i$, hence it is strictly monotone on $J_i$ (by continuity).
\end{proof}

Now we are in a position to show that
for $\{\mu_t\}$ constructed in the proof of Theorem~\ref{t-L1}
it holds that
$[t\mapsto \mu_t] \notin L^\infty((0,T); \mathscr M(\R))$.
We argue by contradiction. Since by \eqref{e-def-b} for a.e. $t$
\begin{equation*}
|\tilde\mu_t| = \#(f^{-1}(t)),
\end{equation*}
the inclusion
$[t\mapsto \mu_t] \in L^\infty((0,T); \mathscr M(\R))$
is equivalent to the inequality
$\esssup_t \#(f^{-1}(t)) < \infty$.

From Lemma~\ref{l-somewhere-monotone} it follows that the function $f$ constructed above is monotone on some nonempty open interval $I\subset (0,1)$. But then $f' \ge 0$ a.e. on $I$, and this contradicts the construction of $f$ (more specifically, the sets $P$ and $N$).

\begin{remark}
  We  remark that if $f$ were monotone on some interval $I$ then uniqueness would fail for the Cauchy problem for \eqref{ODE} with $\b$ constructed in the proof of Theorem~\ref{t-L1}. Indeed, without loss of generality suppose that $f$ is strictly increasing on $I$.
Then for any $x\in I$ there exist at least two (actually, infinitely many) integral curves $\gamma \in \Gamma_\b$ such that $\gamma(0)=x$. Indeed, clearly $\gamma(t):=x$ ($\forall t\in [0,T]$) belongs to $\Gamma_\b$.
On the other hand, for any $y\in I$ such that $y>x$ one can define $\gamma$ by
\begin{equation*}
\gamma(t):=
\begin{cases}
x, & t < f(x); \\
f^{-1}(t), & f(x) \le t < f(y); \\
y, & t \ge f(y).
\end{cases}
\end{equation*}
Then one readily checks that $\gamma \in \Gamma_\b$,
since for a.e. $t\in (f(x), f(y))$ it holds that
\begin{equation*}
\gamma'(t) = \frac{1}{f'(f^{-1}(t))} = \frac{1}{f'(\gamma(t))} = \b(t,\gamma(t)).
\end{equation*}
\end{remark}

\section{Non-uniqueness in the class \texorpdfstring{$L^\infty((0,T); \mathscr M(\R^2))$}{Linf(0,T;M(R2))}} 

The aim of this final section is to show the following result:
  
  \begin{theorem}\label{t-Linfty}
    There exist $T>0$, a bounded Borel $\b \colon [0,T] \times \R^2 \to \R^2$ and $[t\mapsto \mu_t] \in L^\infty((0,T); \mathscr M(\R^2))$ satisfying the following conditions:
    \begin{enumerate}
      \item[(i)] $\b$ has only constant characteristics, i.e. $\gamma\in \Gamma_\b$ if and only if there exists $x\in \R^2$ such that $\gamma(t) = x$ for all $t\in[0,T]$; 
      \item[(ii)] $\{\mu_t\}_{t\in [0,T]}$ solves \eqref{PDE} with zero initial condition.
    \end{enumerate}
  \end{theorem}

\begin{proof}
The proof will consist in essentially two steps. We will first work in 2D, constructing an example very similar to the one discussed for the proof of Theorem \ref{t-L1}. We will then suitably embed this into the three-dimensional euclidean space $\R^3$ in such a way that the path of measures resulting from this construction will be uniformly bounded. 

Consider the three dimensional Euclidean space with the coordinates $(x,y,t)$.
Let $(\xi,\eta,\zeta)$ denote the coordinates in the Cartesian system
with the origin $O'=(2^{-1},2^{-1},0)$ and the axes
$O'\xi$, $O'\eta$ and $O'\zeta$ having directions $(-1,1,0)$, $(-1,-1,2)$ and $(1,1,1)$ respectively (see Fig. \ref{fig:step1}).

\begin{figure}[h!]
  \subfloat[The coordinates $(x,y,t)$ and $(\xi,\eta,\zeta)$. \label{fig:step1}]
  {
      \tdplotsetmaincoords{70}{90}
      \begin{tikzpicture}[scale=2.3,tdplot_main_coords,>=latex, x={(1,-0.5,0)}]
      %
      %
      \draw (-2,0,0)--(-1,0,0);
      \draw[dashed](-1,0,0)--(1,0,0);
      \draw (0,-1.5,0)--(0,-.3,0);
      \draw[dashed](0,-1,0)--(0,1,0);
      \draw[->] (0,1,0)--(0,1.5,0) node[anchor=north east]{$y$};
      \draw (0,0,-1.5)--(0,0,-.3);
      \draw[dashed](0,0,-.2)--(0,0,1);
      \draw[thick,->, red](.5,.5,0) -- (.7,.7,-2*.2) node[anchor=north]{$\bm n_3$};
      \draw[thick,->, red](0,.5,.5) -- (-2*.2,.7,.7) node[anchor=south]{$\bm n_1$};
      \draw[thick,->, red](.5,0,.5) -- (.7,-2*.2,.7) node[anchor=north]{$\bm n_2$};
      \foreach \x in {1,1}{
        \foreach \y in {1,1} {
          \foreach \z in {1,1} {
            \ifthenelse{\x=-1}{
              \filldraw[fill opacity=0.15, draw=blue, fill=blue!20, loosely dashed]
              (0,0,\z)--(0,\y,0)--(\x,0,0)--cycle;
            }{
            \filldraw[fill opacity=0.15, draw=blue, fill=blue!20]
            (0,0,\z)--(0,\y,0)--(\x,0,0)--cycle;
          }
        }
      }
    }
    
    \draw[->] (0,0,1)--(0,0,1.5) node[anchor=north east]{$t$}; 
    \draw[->] (1,0,0)--(2,0,0) node[anchor=north east]{$x$};  
    \draw[thick,->,green!50!black] (0.5,0.5,0)--(0.5-0.707,0.5+0.707,0) node[anchor=south east]{$\xi$};
    \draw[thick,->,green!50!black] (0.5,0.5,0)--(0.5-0.408,0.5-0.408,0+0.816) node[anchor=north east]{$\eta$};
    \draw[very thick,->,green!50!black] (0.5,0.5,0)--(0.5+0.577,0.5+0.577,0+0.577) node[anchor=west]{$\zeta$};
    \end{tikzpicture}
  }\qquad 
\subfloat[Extension of the vector field using reflections. \label{fig:step4}]
{
  \tdplotsetmaincoords{70}{90}
  \begin{tikzpicture}[scale=2.3,tdplot_main_coords,>=latex, x={(1,-0.5,0)}]
  %
  %
    \draw (-2,0,0)--(-1,0,0);
    \draw[dashed](-1,0,0)--(1,0,0);
    \draw (0,-1.5,0)--(0,-.3,0);
    \draw[dashed](0,-1,0)--(0,1,0);
    \draw[->] (0,1,0)--(0,1.5,0) node[anchor=north east]{$y$};
    \draw (0,0,-1.5)--(0,0,-1);
    \draw[dashed](0,0,-1)--(0,0,1);
  
  \foreach \x in {-1,1}{
    \foreach \y in {-1,1} {
      \foreach \z in {-1,1} {
        \ifthenelse{\x=-1}{
          \filldraw[fill opacity=0.3, draw=blue, fill=blue!20, loosely dashed]
          (0,0,\z)--(0,\y,0)--(\x,0,0)--cycle;
        }{
        \filldraw[fill opacity=0.3, draw=blue, fill=blue!20]
        (0,0,\z)--(0,\y,0)--(\x,0,0)--cycle;
      }
    }
  }
}

\draw[->] (0,0,1)--(0,0,1.5) node[anchor=north east]{$t$}; 
\draw[->] (1,0,0)--(2,0,0) node[anchor=north east]{$x$};  
\end{tikzpicture}
}
  \caption{Construction of the vector field $\bm B$ and the function $u$.}
\end{figure}
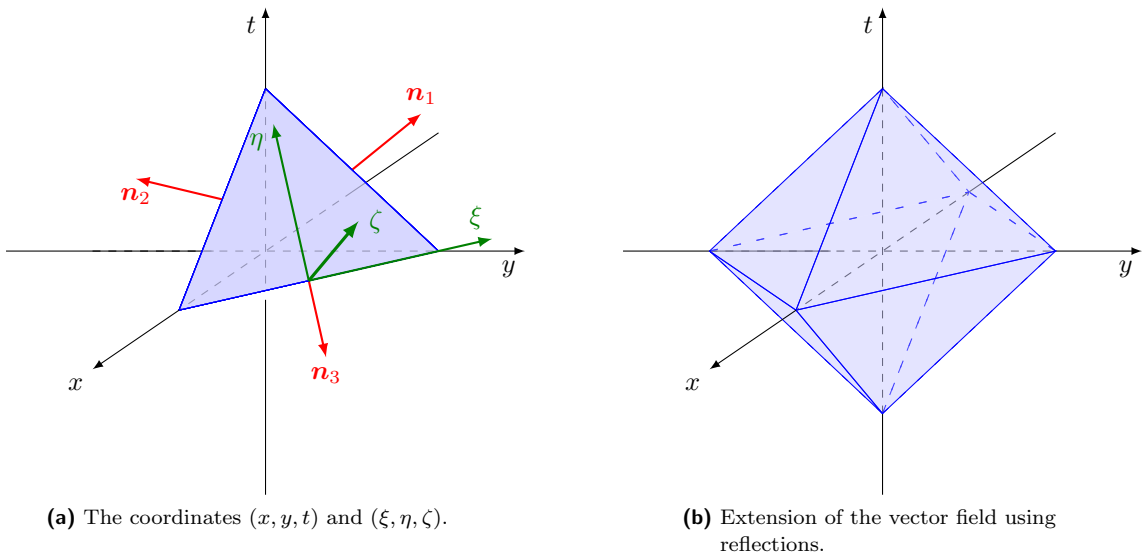

\textit{The 2D construction.}
Let us consider the plane $O'\xi \eta$ and work in the coordinates $(\xi,\eta)$.
Let $f$ and $f_k$ ($k \in \N$) be the functions constructed in the proof of Lemma \ref{l-sets-P-and-N}. We set $P:=f^{-1}(1)$, $N:= \R\setminus P$, $P^k:=(f^{k})^{-1}(1)$ and $N^k:=\R \setminus P^k$. Let
\begin{equation*}
\bm W(\xi, \eta) := \alpha\cdot (1, \1_{P}(\xi) - \1_{N}(\xi)),
\quad
\bm W^k(\xi, \eta) := \alpha\cdot (1, \1_{P^k}(\xi) - \1_{N^k}(\xi)),
\end{equation*}
where $\alpha>0$ is a geometrical constant to be specified later.
Clearly $\dive_{\xi,\eta}(\bm W)=0$ and the $\eta$-component of $\bm W$ (and $\bm W^k$) takes only the values $\pm \alpha$.

Let now $D\subset \R^2_{\xi,\eta}$ be an open, bounded set with piecewise smooth boundary $\partial D$ and assume that $\partial D$ does not contain vertical segments. We claim that
\begin{equation} \label{e-Gauss-Green-corollary}
\dive ( \1_D \bm W) = \bm W \cdot \bm \nu \H^{1}\rest_{\partial D}
\quad \text{in $\mathscr D'(\R^2)$},
\end{equation}
where $\bm \nu$ is the outer unit normal to $\partial D$
and $\H^{1}\rest_{\partial D}$ is the restriction of $\H^1$ to $\partial D$.

Indeed, $\bm W^k$ are piecewise constant inside $D$, so decomposing $D$
into finitely many pieces, applying the classical Gauss--Green Theorem
for each piece and summing the results we get that for any test function $\phi\in C^\infty_c(\R^2)$
\begin{equation*}
\int_{D} \bm W^k \cdot \nabla \phi \, dx = \int_{\partial D} \phi \bm W^k \cdot \bm \nu d\H^{1}. 
\end{equation*}
Since $\partial D$ does not contain vertical segments,
by construction of the sets $P^k$ and $N^k$ (see Lemma~\ref{l-sets-P-and-N}) we have $\bm W^k\to \bm W$ \;
$\H^{1}$-a.e. on $\partial D$ as $k\to \infty$.
Passing to the limit by means of Dominated convergence Theorem 
we get \eqref{e-Gauss-Green-corollary}.

\textit{Passage to 3D.}
We extend $\bm W$ to the whole space using the coordinates $(\xi,\eta,\zeta)$ as follows:
\begin{equation*}
\bm V(\xi, \eta, \zeta) := \alpha\cdot (1, \1_{P}(\xi) - \1_{N}(\xi), 0).
\end{equation*}

Let us switch to the coordinates $(x,y,t)$.
Since $(0,0,1)\cdot \bm V = \pm\alpha \sqrt{2/3}$, fixing $\alpha=\sqrt{3/2}$ we achieve that the $t$-component of $\bm V$ is $\pm 1$.

Let $T:=\{(x,y,t) \;|\; x,y,t>0, \; x+y+t = 1\}$.
By \eqref{e-Gauss-Green-corollary} it holds that $\dive (\1_T \bm V \H^2) = g_1 + g_2 + g_3$, where $g_i = \bm V \cdot \bm n_i \H^1\rest E_i$ and
\begin{equation*}
\begin{aligned}
E_1 = \{(0,y,t) \;|\;   y,t>0, \;   y+t = 1\}&, \quad \bm n_1 = (-2,1,1)/\sqrt{6},\\
E_2 = \{(x,0,t) \;|\; x,  t>0, \; x+  t = 1\}&, \quad \bm n_2 = (1,-2,1)/\sqrt{6},\\
E_3 = \{(x,y,0) \;|\; x,y  >0, \; x+y   = 1\}&, \quad \bm n_3 = (1,1,-2)/\sqrt{6}.
\end{aligned}
\end{equation*}
We define $\bm U\colon \R^3 \to \R^3$ as follows:
\begin{equation*}
\bm U(x,y,t) = \sum_{s_1,s_2,s_3 \in \{\pm 1\}}
\1_T(s_1 x, s_2 y, s_3 t) \, \bm U_{s_1, s_2, s_3}(s_1 x, s_2 y, s_3 t),
\end{equation*}
where
\begin{equation*}
\bm U_{s_1, s_2, s_3}(x,y,t) = 
(s_2 s_3 \bm V_1(x,y,t), \,
 s_1 s_3 \bm V_2(x,y,t), \,
 s_1 s_2 \bm V_3(x,y,t)).
\end{equation*}
Observe that 
\begin{equation}\label{e-div-U}
\dive (\bm U \H^2) = g, \quad\text{where} \quad  g(x,y,t) = \sum_{i=1}^3 \sum_{s_1,s_2,s_3 \in \{\pm 1\}} 
s_1 s_2 s_3 \, g_i(s_1 x, s_2 y,s_3 t)
\end{equation}
(in the sense of distributions). Notice that also
$g_1(x,y,t) = g_1(-x,y,t)$, $g_2(x,y,t) = g_2(x,-y,t)$
and $g_3(x,y,t) = g_3(x,y,-t)$. Because of this symmetry the
right hand side of \eqref{e-div-U} is zero. For instance, for $i=1$
we have
\begin{multline*}
\sum_{s_1,s_2,s_3 \in \{\pm 1\}} 
s_1 s_2 s_3 \, g_1(s_1 x, s_2 y,s_3 t)
\\
=
\sum_{s_2,s_3 \in \{\pm 1\}} 
s_2 s_3 \, g_1(x, s_2 y,s_3 t)
+
\sum_{s_2,s_3 \in \{\pm 1\}} 
(-1) s_2 s_3 \, g_1(-x, s_2 y,s_3 t) = 0.
\end{multline*}

Consider the octahedron $\Delta:= \{(x,y,t) : \bm U(x,y,t) \ne 0\}$ and let
\begin{equation*}
u(x,y,t) := \begin{cases}
\bm U_3(x,y,t), & (x,y,t) \in \Delta; \\
0, & (x,y,t) \notin \Delta,
\end{cases}
\quad
\bm B(x,y,t) := \begin{cases}
u(x,y,t) \bm U(x,y,t), & (x,y,t) \in \Delta; \\
(0,0,1), & (x,y,t) \notin \Delta.
\end{cases}
\end{equation*}

Then $\bm B_3 = 1$ (everywhere) and by \eqref{e-div-U} we have $\dive(\1_\Delta u \bm B \H^2) = \dive(\bm U \H^2) = 0$ (in the sense of distributions).
Hence for any test function $\fhi\in C_c^\infty(\R^3)$
\begin{equation}\label{e-div-uB}
\int_{\Delta} u \bm B \cdot \nabla_{x,y,t} \fhi \, d\H^2 = 0.
\end{equation}
Denoting with $S_t := \{x,y\in \R \;|\; (x,y,t)\in \Delta\}$ and disintegrating
the measure $\H^2 \rest \Delta$ as
\begin{equation*}
\H^2 \rest \Delta = \int \nu_t \, dt,
\quad\text{where}\quad
\nu_t = \alpha \H^1 \rest S_t,
\end{equation*}
(see e.g. \cite[Thm. 2.28]{AFP}) we can rewrite \eqref{e-div-uB} as
\begin{equation*}
\int_{\R} \int_{\R^2}(u \bm B \cdot \nabla_{x,y,t} \fhi) d\nu_t \, dt = 0.
\end{equation*}
Then the family of measures
\begin{equation*}
\mu_t := u \cdot \nu_t
\end{equation*}
satisfy \eqref{PDE} with
\begin{equation*}
\bm b(x,y,t) := (\bm B_1(x,y,t), \bm B_2(x,y,t)).
\end{equation*}

\textit{The characteristics of $\bm b$.}
We claim that $\gamma\in C(\R; \R^2)$ is a characteristic of $\bm b$
if and only if $\gamma(t) = \gamma(0)$ for all $t$.
This claim follows immediately if $\gamma(t)\notin \Delta$ for all $t$
since outside of $\Delta$ the vector field $\bm b$ is zero.
Therefore it is sufficient to show that $\gamma$ can intersect
each face of $\Delta$ at most once. 

Suppose that $\gamma$ intersects the face $T$ (defined above) in two points.
Since $\bm b$ is zero outside of $\Delta$ this is possible only if
there exists some nonempty segment $[a,b]$
such that $(\gamma_1(t), \gamma_2(t), t) \in \Delta$ for all $t\in[a,b]$. Then in the coordinates $(\xi,\eta,\zeta)$ the ODE for $\gamma$ can be written as
\begin{equation*}
\dot \xi = \alpha (\1_P(\xi) - \1_N(\xi)), \quad
\dot \eta = \alpha, \quad
\dot \zeta = 0.
\end{equation*}
But the first equation does not have solutions (see e.g. \cite{Gus18a} for the details), hence we have obtained a contradiction.

\textit{The uniform bounds.}
Ultimately, by definition of $\nu_t$
\begin{equation*}
|\nu_t|(\R^2) = \alpha \cdot 4 \sqrt{2}
\cdot \begin{cases}
1-t, & t\in [0,1];\\
1+t, & t\in [-1,0];\\
0, & t \notin[-1,1],
\end{cases}
\end{equation*}
hence $|\mu_t| \le \alpha \cdot 4 \sqrt{2}$, i.e. the family of measures $\{\mu_t\}$ is uniformly bounded.
\end{proof}


\section{Continuous vector fields}
\label{s:cont_one_dim}

In this section we prove some partial results for continuous vector fields that were
mentioned in the Introduction.

\begin{proposition}\label{p-o+b=cont}
If a vector field $\b\colon (0,T)\times \R^d \to \R^d$ satisfies (\hyperref[i-osgood]{O}) and (\hyperref[i-bounded]{B}) then $\b$ is continuous. 
\end{proposition}

\begin{proof}
Indeed, suppose that for some $t\in (0,T)$, $x\in \R^d$ and $\{x_n\}_{n\in\N}$ it holds that $x_n\to x$ and $\b(x_n) \not\to \b(x)$ as $n\to \infty$. Since $\b$ is uniformly bounded, by passing if necessary to a subsequence we may assume that $\b(x_n)\to \b(x)+z$ for some $z\in \R^d$ as $n\to \infty$. By (\hyperref[i-osgood]{O}) for any $y\in \R^d$
\begin{equation*}
|\langle \b(x_n) - \b(y), x_n - y\rangle| \le C(t) |x_n - y| \rho(|x_n - y|).
\end{equation*}
Passing to the limit in both sides of this inequality we get
\begin{equation*}
|\langle z+\b(x) - \b(y), x - y\rangle| \le C(t) |x - y| \rho(|x - y|).
\end{equation*}
Hence by triangle inequality using (\hyperref[i-osgood]{O}) again we obtain
\begin{equation*}
|\langle z, x-y\rangle|
= |\langle z + \b(x) - \b(y), x-y\rangle - \langle\b(x) - \b(y), x-y\rangle|
\le 2C(t) |x-y|\rho(|x-y|).
\end{equation*}
Taking $y= x + s\cdot z$ with $s>0$ we get
\begin{equation*}
|z|\le 2 C(t) \rho(s |z|).
\end{equation*}
Passing to the limit as $s\to 0$ we get $|z|=0$, and this concludes the proof.
\end{proof}

\begin{proposition}\label{p-1d-cont-uniq}
Suppose that $\b \in C(\R)$ and for any $(t,x)\in \R^2$ there exists a unique $\gamma \in \Gamma_\b$ such that $\gamma(t)=x$. Then for any $\bar\mu \in \mathscr M(\R)$ the Cauchy problem for \eqref{PDE} with the initial condition $\mu_t|_{t=0} = \bar \mu$ has a unique solution
$[t\mapsto \mu_t] \in L^1(0,T; \mathscr M(\R))$.
\end{proposition}

\begin{proof}
Suppose that $[t\mapsto \mu_t] \in L^1(0,T; \mathscr M(\R))$ is a (signed) measure-valued solution to the continuity equation with $\bar \mu=0$. Then there exists a Lebesgue
negligible set $N\subset (0,T)$ such that for all $\tau\in (0,T)\setminus N$ 
for any $\Phi\in C^1_c([0,\tau]\times \R)$ it holds that
\begin{equation}\label{e-Newton-Leibniz}
\int_{\R} \Phi(\tau, x) \, d\mu_\tau(x) - \int_{\R} \Phi(0,x) d\bar \mu(x)
= \int_0^\tau \int_\R \b \cdot \d_x \Phi(t,x) \, d\mu_t(x)\, dt.
\end{equation}
(Indeed, first one can consider finite linear combinations of functions $\Phi$ having the form $\Phi(t,x) = \psi(t)\phi(x)$, where $\phi$ belong to some
countable dense subset of $C^1_c(\R)$ and $\psi \in C^1_c([0,T])$ are arbitrary. 
For such test functions \eqref{e-Newton-Leibniz} follows from \eqref{eq-continuity-distrib},
and in the general case one can apply an approximation argument.)

There are countably many open intervals where $\b>0$ or $\b<0$ (and $\bm b = 0$ on the complement of the union of all those intervals). Consider one of the intervals, i.e. suppose that $\b(\alpha)=\b(\beta) = 0$, $\alpha < \beta$ and $\b>0$ on $(\alpha,\beta)$.
Fix $x_0 \in (\alpha,\beta)$ and for all $x\in(\alpha,\beta)$ let
$$
F(x) := \int_{x_0}^x \frac{dx}{\b(x)}.
$$
(Note that $\frac{1}{\b}\in L^1[x_0, x]$ since $\min_{[x_0, x]} \b > 0$ by continuity.)
Clearly $F \in C^1(\alpha,\beta)$. By uniqueness of integral curves $F(\alpha+0) = -\infty$ and $F(\beta-0) = +\infty$.
Furthermore, $F$ is strictly increasing and continuous, hence $F^{-1} \colon \mathbb R \to (\alpha,\beta)$ is continuous and strictly increasing as well.
Since $F\in C^1(\alpha, \beta)$ we also have $F^{-1}\in C^1(\R)$. Hence
$$
X(t,x) := F^{-1}(F(x) + t)
$$
belongs to $C^1(\R\times (\alpha,\beta))$ by the chain rule. Moreover, $X(\cdot, x)$ solves \eqref{ODE}. 

Let $\omega \in C_c^1(\alpha,\beta)$ be an arbitrary test function and fix $\tau\in(0,T)$.
Then $\varphi(t,x) := \omega(X(\tau-t, x))$ belongs to $C^1_c([0,\tau]\times (\alpha,\beta))$.
(Indeed, if $[u,v]\subset (\alpha,\beta)$ contains the support of~$\omega$, then
the support of~$\varphi$ is contained in $[0,\tau]\times[X(-\tau, u), v]$.)
Moreover, $\varphi$ satisfies the transport equation $\partial_t \varphi + \b \partial_x \varphi = 0$ (pointwise) with the final condition $\varphi(\tau, x) = \omega(x)$.

Using the test function $\Phi = \varphi$ in \eqref{e-Newton-Leibniz} we get
$$
\int_\R \omega(x) \, d\mu_\tau(x) = 0
$$
and by arbitrariness of $\omega$ this implies that $\mu_\tau \equiv 0$ (for all $\tau \in (0,T)\setminus N$).
Since this holds for any of the intervals where $\b$ has constant sign, we have thus proved that $\mu_t$ is concentrated on $\{\b=0\}$ and then it solves \eqref{PDE} with $\b \equiv 0$. Hence $\mu_t \equiv 0$ globally.
\end{proof}

\section*{Acknowledgements}

The first author acknowledges ERC Starting Grant 676675 FLIRT. 
The work of the second author was supported by the ``RUDN University Program 5-100''
and RFBR Grant 18-31-00279.

\bibliographystyle{alpha}
\bibliography{biblio}
\end{document}